\documentclass[11pt]{amsart}
\usepackage{amsmath, amssymb}
\usepackage{amsfonts}
\usepackage[arrow,matrix,curve,cmtip,ps]{xy}

\usepackage{amsthm}

\allowdisplaybreaks

\newtheorem{theorem}{Theorem}[section]
\newtheorem{lemma}[theorem]{Lemma}
\newtheorem{proposition}[theorem]{Proposition}
\newtheorem{corollary}[theorem]{Corollary}

\newtheorem*{theorem*}{Theorem}
\theoremstyle{remark}
\newtheorem{remark}[theorem]{Remark}
\newtheorem{definition}[theorem]{Definition}

\numberwithin{equation}{section}

\newcommand{\Z}{\mathbb{Z}}
\newcommand{\N}{\mathbb{N}}
\newcommand{\Q}{\mathbb{Q}}
\newcommand{\R}{\mathbb{R}}
\newcommand{\C}{\mathbb{C}}

\begin{document}
\title[Equiangular Tight Frames from Complex Seidel Matrices]{Equiangular Tight Frames from Complex  Seidel Matrices Containing Cube Roots of Unity}

\author{Bernhard G. Bodmann}

\address{Department of Mathematics \\ University of Houston \\ Houston, TX 77204-3008 \\USA}
\email{bgb@math.uh.edu}

\author{Vern I. Paulsen}

\address{Department of Mathematics \\ University of Houston \\ Houston, TX 77204-3008 \\USA}
\email{vern@math.uh.edu}

\author{Mark Tomforde}

\address{Department of Mathematics \\ University of Houston \\ Houston, TX 77204-3008 \\USA}
\email{tomforde@math.uh.edu}

\thanks{The first author was supported by
NSF Grant DMS-0807399, the second author by NSF Grant DMS-0600191}

\date{\today}



\begin{abstract}
We derive easily verifiable conditions which characterize when complex Seidel
matrices containing cube roots of unity have exactly two eigenvalues. The existence 
of such matrices is equivalent to the existence of equiangular tight frames for which
the inner product between any two frame vectors is always a common multiple 
of the cube roots of unity. We also exhibit a relationship between these equiangular 
tight frames, complex Seidel matrices, and highly regular, directed graphs. We construct
examples of such frames with arbitrarily many vectors.
\end{abstract}

\maketitle

\section{Introduction}

Equiangular tight frames play an important role in several areas of mathematics, ranging 
from signal processing (see, e.g.\ \cite{Ald93,Cas00,CK07a,CK07b} and references therein) to 
quantum computing
(see, e.g.\ \cite{bodmann-kribs-paulsen:2007,GR08}).
In comparison with the host of results on the construction of equiangular tight frames for finite dimensional real Hilbert spaces \cite{vLS66,Sei68,GS70,BvLS73,HP04},
relatively few means 
are known for constructing equiangular frames in the complex case (see, e.g. \cite{Kal06,Ren07,GR08}). 

The problem of the existence of equiangular frames is known to be equivalent to the 
existence of a certain type of matrix called a \textit{Seidel matrix} \cite{LS73} or {\em signature matrix} \cite{HP04}
with two eigenvalues.
  A matrix $Q$ is a Seidel matrix provided that it is self-adjoint, its diagonal entries are $0$, and its off-diagonal entries are all of modulus one. 
  In the real case, these off-diagonal entries must all be 
  $\pm 1;$ such matrices can then be interpreted as (Seidel) adjacency matrices of graphs.
  The case when these adjacency matrices have two eigenvalues has been characterized in graph-theoretic terms \cite{Sei76,ST81}, see
  also \cite{Moo95}, and
   a fairly complete catalog of all such graphs, when the number of vertices is small, is known.
  Moreover, such graphs 
   are known to exist for an arbitrarily large number of vertices \cite{CalderbankCameronKantorSeidel:1997}.   

In this paper, we study the existence and construction of Seidel matrices with two eigenvalues and off-diagonal entries that are all cube roots of unity. We derive necessary and sufficient conditions for the existence of such matrices and construct examples. In particular, we find arbitrarily large Seidel matrices of this type. We show that the existence of such Seidel matrices is equivalent to the existence of certain highly regular directed graphs, and hence, show that such directed graphs exist on an arbitrarily large number of vertices.

This paper is organized as follows.
We complete the introduction with a more detailed description of equiangular tight frames, Seidel matrices and their relationship. In Section 2, we derive conditions for the existence of such matrices. In Section 3, we refine these conditions and use our relations, in Section 4, to determine the possible sizes of such Seidel matrices for $n \le 100.$ Section 5 discusses the connections with directed graphs.   Section 6 contains examples and Section 7 gives a method for constructing examples of arbitrarily large size.

\subsection{Equiangular  Tight Frames}

Let $\mathcal H$ be a real or complex Hilbert space.  A finite family of vectors $\{ f_1, \ldots, f_n \}$ 
 is called a \emph{frame} provided that  there exist strictly positive real numbers 
$A$ and $B$ such that $$A \| x \|^2 \leq \sum_{i=1}^n | \langle x, f_j \rangle |^2 \leq B \| x \|^2 
\qquad \text{ for all $x \in \mathcal H$} \, .$$  
A frame is said to be a \emph{tight frame} 
if we can choose $A=B$. When $A=B=1,$ then the frame is called a {\em normalized tight frame} or a {\em Parseval frame.} Replacing $f_i$ by $f_i/\sqrt{A}$ always normalizes a tight frame.  

One can show (see \cite[p.~21]{Cas00}, for example) that 
a family $\{f_1, \ldots, f_n \}$ is a tight frame with constant $A$ if and only if 
\begin{equation} \label{Parseval-Id} 
         x = \frac{1}{A} \sum_{i=1}^n \langle x, f_i \rangle f_i \qquad \text{ for all $x \in \mathcal H$.}
\end{equation} 

There is a natural equivalence relation for tight frames, motivated by
simple operations on the frame vectors which preserve identity (\ref{Parseval-Id}).

We say that two tight frames $\{f_1, f_2, \dots f_n\}$ and
$\{g_1, g_2 , \dots g_n\}$ for $\mathcal H$ are \emph{unitarily equivalent}
if there exists a unitary operator $U$ on $\mathcal H$ such that for all $i\in \{1,2 ,\dots n\}$, 
$g_i = U f_i$. We say that they are \textit{switching
equivalent} if
there exist a unitary operator $U$ on $\mathcal H$, a permutation $\pi$
on $\{1, 2, \dots n\}$ and a family of unimodular constants $\{\lambda_1, \lambda_2, \dots \lambda_n\}$
such that for all $i\in \{1,2 ,\dots n\}$, $g_i = \lambda_i U f_{\pi(i)}$. 
If $\mathcal H$ is a real Hilbert space, $U$ is understood to be orthogonal,
and all $\lambda_i \in \{\pm 1\}$.

In this paper we shall be concerned only with Parseval frames for the $k$-dimen\-sio\-nal complex Hilbert space 
$\mathbb C^k$, equipped with the canonical inner product. We use the term \emph{$(n,k)$-frame} to 
mean a Parseval frame of $n$ vectors for $\C^k$.

Every such Parseval frame gives rise to an isometric embedding of $\C^k$ into $\C^n$ via the map
$$
   V: \C^k \to \C^n,  \, (Vx)_j =  \langle x, f_j\rangle \mbox{ for all } j\in \{1, 2, \dots n\} \, 
$$ which is called the {\em analysis operator} of the frame.
Because $V$ is linear, we may identify $V$ with an $n \times k$ matrix and the vectors $\{ f_1, \ldots, f_n \}$ are the respective columns of $V^*.$  Conversely, given any $n \times k$ matrix $V$ that defines an isometry, if we let $\{ f_1, \ldots, f_n \}$ denote the columns of $V^*,$ then this set is an $(n,k)$-frame and $V$ is the analysis operator of the frame.  

If $V$ is the analysis operator of an $(n,k)$-frame, then since $V$ is an isometry, we see that $V^*V = I_k$ and the $n \times n$ matrix $VV^*$ is a self-adjoint projection of rank $k.$  Note that $VV^*$ has entries $(VV^*)_{ij} = ( \langle f_j, f_i \rangle )$.  Thus, $VV^*$ is the \emph{Grammian matrix} (or \emph{correlation matrix}) of the set of vectors. Conversely, any time we have an $n \times n$ self-adjoint projection $P$ of rank $k$, we can always factor it as $P=VV^*$ for some $n \times k$ matrix $V$.  In this case we have $V^*V = I_k$ and hence $V$ is an isometry and the columns of $V^*$ are an $(n,k)$-frame.  Moreover, if $P = WW^*$ is another factorization of $P$, then there exists a unitary $U$ such that $W^*= UV^*$, and the frame corresponding to $W$ differs from the frame corresponding to $V$ by applying the same unitary to all frame vectors, which is included in our equivalence relation. However, switching equivalence 
is coarser than just identifying all frames with the same Grammian.
In \cite{HP04} it is shown that it corresponds to
identifying frames for which the Grammians can be obtained from each other by conjugation with diagonal unitaries
and permutation matrices.

\begin{definition}
An $(n,k)$-frame $\{f_1, \ldots, f_n \}$ is called {\em uniform} if there is a constant $u>0$ such that $\|f_i \|=u$ for all $i.$   An $(n,k)$-frame is called {\em equiangular} if all of the frame vectors are non-zero and the angle between the lines generated by any pair of frame vectors is a constant, 
that is, provided that there is a constant $b$ such that $| \langle f_i/\|f_i\|, f_j/\|f_j\| \rangle |=b$ for all $i \neq j$.  
\end{definition}

Many places in the literature define equiangular to mean that the $(n,k)$-frame is uniform and that there is a constant $c$ so that $| \langle f_i, f_j \rangle |= c$ for all $i \neq j.$ However, the assumption that the frame is uniform is not needed in our definition as the following result shows.

\begin{proposition} Let $\{ f_1, \ldots , f_n \}$ be a tight frame for 
$\C^k.$ If all frame vectors are non-zero and if there is a constant $b$ so that $| \langle f_i/\|f_i\|, f_j/\|f_j\| \rangle | =b$ for all $i \neq j,$ then $\|f_i\|= \|f_j\|$ for every $i$ and $j.$
\end{proposition}
\begin{proof}
Without loss of generality, we may assume that the frame is a Parseval frame, so that $P= ( \langle f_j, f_i \rangle)_{i,j=1}^n$ is a projection of rank $k$. Hence, $P=P^2$ and so upon equating the $(i,i)$-th entry and using the fact that the trace of $P$ is $k$, we see that
$\|f_i\|^2 = \langle f_i, f_i \rangle = \sum_{j=1}^n \langle f_j,f_i \rangle \langle f_i, f_j \rangle = \|f_i\|^4 + \sum_{j \neq i}^n b^2\|f_i\|^2\|f_j\|^2
= \|f_i\|^4 + b^2\|f_i\|^2(k - \|f_i\|^2),$
which shows that $\|f_i\|^2$ is a (non-zero) constant independent of~$i.$
\end{proof}

In \cite{HP04} a family of $(n,k)$-frames was introduced that was called {\em $2$-uniform frames}. It was then proved that a Parseval 
frame is $2$-uniform if and only if it is equiangular. Thus, these terminologies are interchangeable in the literature, but the equiangular terminology has become more prevalent.

\subsection{Seidel Matrices and Equiangular Tight Frames}

At this point we revisit an approach that has been used to construct equiangular tight frames \cite{LS73}. The previous section shows that an $(n,k)$-frame is determined up to unitary equivalence by
its Grammian matrix.  This reduces the problem of constructing an
$(n,k)$-frame to constructing an $n \times n$ self-adjoint projection $P$ of rank $k$.  

If an $(n,k)$-frame $\{f_1, f_2, \dots f_n\}$ is uniform, then it is known that $\|f_i\|^2= k/n$ for all
$i \in \{1, 2, \dots n\}$.
It is shown in \cite[Theorem~2.5]{HP04} 
that if $\{ f_1, \ldots, f_n \}$ is an equiangular $(n,k)$-frame, then for all $i \neq j$, 
$|\langle f_j, f_i \rangle | = c_{n,k} = \sqrt{\frac{k(n-k)}{n^2(n-1)}}$.
Thus we may write $$VV^* = (k/n) I_n + c_{n,k} Q$$ where 
$Q$ is a self-adjoint $n\times n$ matrix satisfying $Q_{ii} = 0$ for all $i$ and $|Q_{ij} | = 1$ for all $i \neq j$.  
This matrix $Q$ is called the \emph{Seidel matrix} or \emph{signature matrix} associated with the $(n,k)$-frame.

The following theorem characterizes the signature matrices of equiangular $(n,k)$-frames.

\begin{theorem}[Theorem~3.3 of \cite{HP04}] \label{2-uniform-char}
Let $Q$ be a self-adjoint $n \times n$ matrix with $Q_{ii}=0$ and $|Q_{ij}|=1$ for all $i \neq j$.  Then the following are equivalent:
\begin{itemize}
\item[(a)] $Q$ is the signature matrix of an equiangular $(n,k)$-frame for some $k$;
\item[(b)] $Q^2 = (n-1) I + \mu Q$ for some necessarily real number $\mu$; and
\item[(c)] $Q$ has exactly two eigenvalues.
\end{itemize}
\end{theorem}

This result reduces the problem of constructing equiangular $(n,k)$-frames to the problem of constructing Seidel matrices
with two eigenvalues.  
In particular, Condition~(b) is particularly useful since it gives an easy-to-check condition to verify that a 
matrix $Q$ is the signature matrix of an equiangular tight frame.

Furthermore, if $Q$ is a matrix satisfying any of the three equivalent conditions 
in Theorem~\ref{2-uniform-char}, and if $\lambda_1 < 0 < \lambda_2$ are its two eigenvalues, then 
the parameters $n$, $k$, $\mu$, $\lambda_1$, and $\lambda_2$ satisfy
the following properties:

\begin{align}
\mu &= (n-2k) \sqrt{\frac{n-1}{k(n-k)}} = \lambda_1 + \lambda_2, \qquad k = \frac{n}{2} - \frac{\mu n}{2 \sqrt{4(n-1) + \mu^2}}  \label{param-eqs} \\
 \lambda_1&= -\sqrt{\frac{k(n-1)}{n-k}}, \qquad \lambda_2 = \sqrt{\frac{(n-1)(n-k)}{k}}, \qquad n = 1 - \lambda_1 \lambda_2. \notag
\end{align}

\noindent These equations follow from the results in \cite[Proposition~3.2]{HP04} and \cite[Theorem~3.3]{HP04}, and by solving for $\lambda_1$ and $\lambda_2$ from the given equations.

In the case when the entries of $Q$ are all real, we have that the diagonal entries of $Q$ are $0$ and the off-diagonal entries of $Q$ are $\pm 1$.  These matrices can be seen to be Seidel adjacency matrices of a graph \cite{Sei76} on $n$ vertices, and it has been proven that in the real case there is a one-to-one correspondence between the switching equivalence classes of real equiangular tight frames and regular two-graphs \cite[Theorem~3.10]{HP04}.

In a similar vein, we now apply switching equivalence to complex Seidel matrices, in order to derive easily
verifiable conditions which characterize when they have two eigenvalues.

\begin{definition}
Two Seidel matrices $Q$ and $Q'$ are
\emph{switching equivalent} if they can be obtained
from each other by conjugating with a diagonal
unitary and a permutation matrix.

If $Q$ is a Seidel matrix, we say that $Q$ is in a
\emph{standard form} if its first row and column contains
only $1$'s except on the diagonal,
 as shown in \eqref{Q-form} below.  
We
say that it is \emph{trivial} if it has a
standard form which has all of its off-diagonal entries equal to $1$
and \emph{nontrivial} if at least one off-diagonal entry is not equal
to $1.$
\end{definition}

One can verify by conjugation with an appropriate diagonal unitary 
that the equivalence class of any Seidel matrix contains a matrix
of standard form, so we only need
to examine when matrices of this form
have two eigenvalues. 

Since, in the real case, the off-diagonal entries of $Q$ are in the set $\{ -1, 1 \}$, it seems promising in the complex 
situation to harness combinatorial techniques and consider the case when the off-diagonal entries of $Q$ are $m$\textsuperscript{th} 
roots of unity for some $m \geq 3$. We conjecture that these cases will give a description of families of complex $2$-uniform 
frames in analogy with the characterization that has been obtained in the real case.  The purpose of this paper is to examine the 
simplest complex case (when $m=3$), and to construct new frames in this setting.

\section{Signature matrices with entries in the cube roots of unity}

In this section we consider nontrivial signature matrices whose off-diagonal entries are cube roots of unity.  We obtain a number of necessary and sufficient conditions for such a signature matrix of an equiangular $(n,k)$-frame to exist.  These results are useful because they allow us to rule out many values of $n$ and $k$ in the  search for examples of such frames. 

Let $\omega = -\frac{1}{2} + i \frac{\sqrt{3}}{2}$.  Then the set $\{ 1, \omega, \omega^2 \}$ is the set of cube roots of unity.  Note also that $\omega^2 = \overline{\omega}$ and $1 + \omega + \omega^2 = 0$.  

\begin{definition}
We call a matrix $Q$ a \emph{cube root Seidel matrix} 
if it is self-adjoint, has vanishing diagonal entries, and off-diagonal entries which are all cube roots of unity.
If $Q$ has exactly two eigenvalues, then we say that it is the \emph{cube root signature matrix} of an equiangular tight frame.
\end{definition}

All equivalence classes of cube root signature matrices contain  representatives in standard form.

\begin{lemma} \label{permutation-lemma}
If $Q'$ is an $n \times n$ cube root Seidel matrix, then it is switching equivalent to a cube root Seidel matrix of the form
\begin{equation}
\label{Q-form}
Q = \begin{pmatrix} 0 & 1 & \cdots & \cdots & 1 \\ 1 & 0 & * & \cdots & * \\ \vdots & * & \ddots & \ddots & \vdots \\ \vdots & \vdots & \ddots & \ddots & *  \\ 1 & * & \cdots & * & 0 \end{pmatrix}
\end{equation}
where the $*$'s are cube roots of unity. Moreover, $Q'$ is the signature matrix of an equiangular $(n,k)$-frame if and only if
$Q$ is the signature matrix of an equiangular $(n,k)$-frame. 
\end{lemma}

\begin{proof}
Suppose that $Q'$ is an $n \times n$ cube root Seidel matrix. Then $Q'$ is self-adjoint with $|Q_{ij}|=1$ for $i \neq j$, and by Theorem~\ref{2-uniform-char} we have that $(Q')^2 = (n-1) I + \mu Q'$ for some real number $\mu$.  If we let $U$ be the diagonal matrix $$U := \begin{pmatrix} 1 & & & & \\ & Q'_{12} & & & \\ & & Q'_{13} & & \\ & & & \ddots & \\ & & & & Q'_{1n} \end{pmatrix}$$ then $U$ is a unitary matrix (since $|Q'_{ij}| =1$ when $i \neq j$), and we see that $Q := U^* Q' U$ is a self-adjoint $n \times n$ matrix with $Q_{ii} = 0$ and $|Q_{ij}|=1$ for $i \neq j$. We see that the off-diagonal elements of $Q$ are cube roots of unity and $Q$ has the form shown in \eqref{Q-form}.  (To see that the off-diagonal elements in the first row and column are $1$'s, recall that $Q'_{ij} = \overline{Q'_{ji}}$.)  Thus $Q$ is a cube root Seidel matrix that is unitarily equivalent to $Q'$. Since $Q$ and $Q'$ have the same eigenvalues, if one of them
is the signature matrix of an equiangular $(n,k)$-frame, then the same holds for the other matrix.
\end{proof}

Next, we present the characterization of cube root Seidel matrices with exactly two eigenvalues
after an elementary insight.

\begin{lemma} \label{1-lemma}
If $a, b, c \in \R$ and $a1 + b\omega + c\omega^2 =0$, then $a=b=c$.
\end{lemma}

\begin{proof}
If $a1 + b\omega + c\omega^2 =0$, then $a1 + b\left(-\frac{1}{2} + i \frac{\sqrt{3}}{2}\right) + c\left(-\frac{1}{2} - i \frac{\sqrt{3}}{2}\right) =0$, and hence $\left(a-\frac{b}{2}-\frac{c}{2}\right) + i\frac{\sqrt{3}}{2}\left(b-c\right) =0$.  It follows that $a-\frac{b}{2}-\frac{c}{2} = 0$ and $b-c = 0$.  Thus $a=b=c$.
\end{proof}

\begin{proposition} \label{number-omega-row-col}
Let $Q$ be a cube root Seidel matrix in standard form, and suppose $Q$ satisfies the equation $$Q^2 = (n-1)I + \mu Q.$$  Then $e := \frac{n-\mu-2}{3}$ is an integer, and for any $j$ with $2 \leq j \leq n$, the 
$j$\textsuperscript{th} column of $Q$ (and likewise the $j$\textsuperscript{th} row) contains $e$ entries equal to $\omega$, contains $e$ entries equal to $\omega^2$, and contains $e+\mu+1=\frac{n+2\mu+1}{3}$ entries equal to $1$.  
\end{proposition}

\begin{proof}
For $2 \leq j \leq n$ define 
\begin{align*}
x_j &:= \# \{ i : Q_{ij} = 1 \} \\
y_j &:= \# \{ i : Q_{ij} = \omega \} \\
z_j &:= \# \{ i : Q_{ij} = \omega^2 \}.
\end{align*}
Since the $j$\textsuperscript{th} column of $Q$ has $n-1$ nonzero entries (recall the zero on the diagonal) we have
\begin{equation} \label{a-b-c-eq}
x_j + y_j + z_j = n-1.
\end{equation}
Also, since $Q^2 = (n-1)I + \mu Q$ we see that for $j \geq 2$ we have $$\mu = \mu Q_{1j} =  [(n-1)I + \mu Q]_{1j} = Q^2_{1j} = (x_j-1) 1 + y_j \omega + z_j \omega^2.$$  Thus $(x_j-\mu-1)1 + y_j \omega + z_j \omega^2 = 0$, and by Lemma~\ref{1-lemma} we have 
\begin{equation} \label{a-b-c-eq-2}
x_j-\mu-1=y_j=z_j.
\end{equation}
Thus \eqref{a-b-c-eq} becomes $x_j + 2 (x_j - \mu -1) = n-1$ so that
\begin{equation} \label{x-eqn}
x_j = \frac{n+2\mu+1}{3}.
\end{equation}
and \eqref{a-b-c-eq-2} gives
\begin{equation} \label{y-z-eqn}
y_j=z_j=\frac{n-\mu-2}{3}.
\end{equation}

Since the quantities in \eqref{x-eqn} and \eqref{y-z-eqn} do not depend on $j$ they are valid for any column.  In addition, since $Q=Q^*$ and $\overline{\omega} = \omega^2$, the same equations hold for the rows of $Q$.
\end{proof}

To summarize the consequences of Proposition~\ref{number-omega-row-col}: In the first column of $Q$ there are $n-1$ entries equal to 1 and one entry equal to $0$.  In the $j$\textsuperscript{th} column (for $j \geq 2$) there are $x_j = \frac{n+2\mu+1}{3}$ entries equal to 1, there are $y_j = \frac{n-\mu-2}{3}$ entries equal to $\omega$, there are $z_j = \frac{n-\mu-2}{3}$ entries equal to $\omega^2$, and there is one entry (on the diagonal) equal to $0$.  Note that these values do not depend on the value of $j$ and the same equations hold for the rows.  In particular, if $Q$ is trivial, then $k=1$, and its off-diagonal entries are all $1$'s, $x_j=n-1$ and $y_j=z_j=0$ for $j \ge 2$.

\section{Equations for nontrivial cube root signature matrices} \label{Eqs-sec}

Suppose that $Q$ is a nontrivial cube root Seidel matrix.  Also suppose that $Q$ is in standard form and that $Q$ satisfies $Q^2 = (n-1)I + \mu Q$.  Since $Q$ is nontrivial we have that $Q_{ij} = \omega$ for some $2 \leq i,j \leq n$ with $i \neq j$.  For these values of $i$ and $j$ let 
\begin{align*}
\alpha &= \# \{ k : Q_{ik} = \omega \text{ and } Q_{kj} = \omega^2 \} \\
\beta &= \# \{ k : Q_{ik} = \omega \text{ and } Q_{kj} = \omega \} \\
\gamma &= \# \{ k : Q_{ik} = \omega \text{ and } Q_{kj} = 1 \} \\
a &= \# \{ k : Q_{ik} = \omega^2 \text{ and } Q_{kj} = \omega^2 \} \\
b &= \# \{ k : Q_{ik} = \omega^2 \text{ and } Q_{kj} = \omega \} \\
c &= \# \{ k : Q_{ik} = \omega^2 \text{ and } Q_{kj} =1 \} \\
A &= \# \{ k : Q_{ik} = 1 \text{ and } Q_{kj} = \omega^2 \} \\
B &= \# \{ k : Q_{ik} = 1 \text{ and } Q_{kj} = \omega \} \\
C &= \# \{ k : Q_{ik} = 1 \text{ and } Q_{kj} = 1 \}.
\end{align*}
We shall now establish equations relating these nine values.  To begin, we see that the number of $\omega$'s in row $i$ is equal to $\alpha + \beta + \gamma+1$. (The $+1$ comes from the term $Q_{ij}=\omega$.)  Also, the number of $\omega$'s in row $i$ is equal to $e := \frac{n-\mu-2}{3}$ by Proposition~\ref{number-omega-row-col}.  Thus 
\begin{equation} \label{EQ-1}
\alpha + \beta + \gamma = e -1.
\end{equation}
In addition, the number of $\omega^2$'s in row $i$ is equal to $a+b+c$, and by Proposition~\ref{number-omega-row-col} the number of $\omega^2$'s in row $i$ is equal to $e$.  Thus
\begin{equation} \label{EQ-2}
a+b+c = e.
\end{equation}
Also, the number of $1$'s in row $i$ is equal to $A + B + C$, and by Proposition~\ref{number-omega-row-col} the number of $1$'s in row $i$ is equal to $e+\mu+1$.    Hence
\begin{equation} \label{EQ-3}
A+B+C = e+\mu +1.
\end{equation}
Next we turn our attention to the $j$\textsuperscript{th} column.  We see that the number of $\omega^2$'s in column $j$ is equal to $\alpha + a + A$, and by Proposition~\ref{number-omega-row-col} the number of $\omega^2$'s in column $j$ is equal to $e$.  Thus 
\begin{equation} \label{EQ-4}
\alpha + a + A = e.
\end{equation}
In addition, the number of $\omega$'s in column $j$ is equal to $\beta + b + B +1$. (The $+1$ comes from the term $Q_{ij}=\omega$.)  Also, by Proposition~\ref{number-omega-row-col} the number of $\omega$'s in column $j$ is equal to $e$.  Thus 
\begin{equation} \label{EQ-5}
\beta + b + B = e -1.
\end{equation}
Furthermore, the number of $1$'s in column $j$ is equal to $\gamma + c + C$, and by Proposition~\ref{number-omega-row-col} the number of $1$'s in column $j$ is equal to $e+\mu+1$.    Hence
\begin{equation} \label{EQ-6}
\gamma + c + C = e+\mu +1.
\end{equation}
Finally, since $Q^2 = (n-1)I + \mu Q$ we have that
\begin{align*}
\mu \omega &= \mu Q_{ij} = [(n-1)I + \mu Q]_{ij} = Q^2_{ij} = \sum_{k=1}^n Q_{ik} Q_{kj} \\
&= \alpha (\omega \omega^2) + \beta (\omega \omega) + \gamma (\omega 1) + a (\omega^2 \omega^2) + b (\omega^2 \omega) + c (\omega^2 1) \\
& \qquad \qquad \qquad \qquad \qquad \qquad \qquad \qquad \qquad + A (1 \omega^2) + B (1 \omega) + C (1)( 1) \\
&= \alpha 1  + \beta \omega^2 + \gamma \omega + a \omega + b 1 + c \omega^2 + A \omega^2 + B \omega + C 1 \\
&= (\alpha+ b + C)1 + (\gamma + a + B)\omega + (\beta + c + A)\omega^2
\end{align*}
so that $$(\alpha+ b + C)1 + (\gamma + a + B - \mu)\omega + (\beta + c + A)\omega^2 = 0.$$  It follows from Lemma~\ref{1-lemma} that $\alpha+ b + C = \gamma + a + B - \mu = (\beta + c + A)$ and thus
\begin{equation} \label{EQ-7}
\alpha - \gamma -a + b -B + C = - \mu
\end{equation}
and 
\begin{equation} \label{EQ-8}
\alpha - \beta + b -c -A +C= 0.
\end{equation}

By looking at \eqref{EQ-1}, \eqref{EQ-2}, \eqref{EQ-3}, \eqref{EQ-4}, \eqref{EQ-5}, \eqref{EQ-6}, \eqref{EQ-7}, and \eqref{EQ-8}, we have eight equations in the nine unknowns $\alpha, \beta, \gamma, a, b, c, A, B, C$.  These equations are not linearly independent: we see that  $\eqref{EQ-1}+\eqref{EQ-2}+\eqref{EQ-3}=\eqref{EQ-4}+\eqref{EQ-5}+\eqref{EQ-6}$.  However, this is the only relation, and when we row reduce this system we obtain the seven equations
\begin{align}
\alpha - B &= -\frac{2\mu}{3} - \frac{1}{3} \tag{Eq.~1} \\
\beta - C &= -\frac{2\mu}{3} - \frac{4}{3} \tag{Eq.~2} \\
\gamma + B + C &= \frac{n}{3} + \mu \tag{Eq.~3} \\
 a - C &= -\frac{\mu}{3} - \frac{2}{3} \tag{Eq.~4} \\
 b + B + C &= \frac{n}{3} + \frac{\mu}{3} - \frac{1}{3} \tag{Eq.~5} \\
 c - B &= -\frac{\mu}{3} + \frac{1}{3} \tag{Eq.~6} \\
 A + B + C &=  \frac{n}{3} + \frac{2\mu}{3} + \frac{1}{3} \tag{Eq.~7}  
\end{align}
with $B$ and $C$ as the two free variables.

It is important to note that the above variables, $\alpha, \beta, \gamma, a, b, c, A, B, C$, really should carry a subscript $(i,j)$ since their actual values could depend on the particular $(i,j)$ that we choose satisfying $Q_{i,j}= \omega$ and we are only asserting that for each such pair $(i,j)$ these equations must be met, not that their values are independent of the pair that we have chosen. 

Similarly, we derive equations for the case that $Q_{ij} = 1$ for some $2 \leq i,j \leq n$ with $i \neq j.$ Later, we shall prove that such an entry always exists.  For these values of $i$ and $j$ let 
\begin{align*}
\alpha^{\prime} &= \# \{ k : Q_{ik} = \omega \text{ and } Q_{kj} = \omega^2 \} \\
\beta^{\prime} &= \# \{ k : Q_{ik} = \omega \text{ and } Q_{kj} = \omega \} \\
\gamma^{\prime} &= \# \{ k : Q_{ik} = \omega \text{ and } Q_{kj} = 1 \} \\
a^{\prime} &= \# \{ k : Q_{ik} = \omega^2 \text{ and } Q_{kj} = \omega^2 \} \\
b^{\prime} &= \# \{ k : Q_{ik} = \omega^2 \text{ and } Q_{kj} = \omega \} \\
c^{\prime} &= \# \{ k : Q_{ik} = \omega^2 \text{ and } Q_{kj} =1 \} \\
A^{\prime} &= \# \{ k : Q_{ik} = 1 \text{ and } Q_{kj} = \omega^2 \} \\
B^{\prime} &= \# \{ k : Q_{ik} = 1 \text{ and } Q_{kj} = \omega \} \\
C^{\prime} &= \# \{ k : Q_{ik} = 1 \text{ and } Q_{kj} = 1 \}.
\end{align*}
We shall now establish equations relating these nine values.  To begin, we see that the number of $\omega$'s in row $i$ is equal to $\alpha^{\prime} + \beta^{\prime} + \gamma^{\prime}$.  Also, the number of $\omega$'s in row $i$ is equal to $e$ by Proposition~\ref{number-omega-row-col}.  Thus 
\begin{equation} \label{EQ-1'}
\alpha^{\prime} + \beta^{\prime} + \gamma^{\prime} = e.
\end{equation}

In addition, the number of $\omega^2$'s in row $i$ is equal to $a^{\prime}+b^{\prime}+c^{\prime}$, and by Proposition~\ref{number-omega-row-col} the number of $\omega^2$'s in row $i$ is equal to $e := \frac{n-\mu-2}{3}$.  Thus
\begin{equation} \label{EQ-2'}
a^{\prime}+b^{\prime}+c^{\prime} = e.
\end{equation}
Also, the number of $1$'s in row $i$ is equal to $A^{\prime} + B^{\prime} + C^{\prime}+1$.  (The +1 comes from the $(i,j)$-entry.)  By Proposition~\ref{number-omega-row-col} the number of $1$'s in row $i$ is equal to $e+\mu+1$.    Hence
\begin{equation} \label{EQ-3'}
A^{\prime}+B^{\prime}+C^{\prime} = e+\mu.
\end{equation}

Next we turn our attention to the $j$\textsuperscript{th} column.  We see that the number of $\omega^2$'s in column $j$ is equal to $\alpha^{\prime} + a^{\prime} + A^{\prime}$, and by Proposition~\ref{number-omega-row-col} the number of $\omega^2$'s in column $j$ is equal to $e$.  Thus 
\begin{equation} \label{EQ-4'}
\alpha^{\prime} + a^{\prime} + A^{\prime} = e.
\end{equation}
In addition, the number of $\omega$'s in column $j$ is equal to $\beta^{\prime} + b^{\prime} + B^{\prime}$.  Also, by Proposition~\ref{number-omega-row-col} the number of $\omega$'s in column $j$ is equal to $e$.  Thus 
\begin{equation} \label{EQ-5'}
\beta^{\prime} + b^{\prime} + B^{\prime} = e.
\end{equation}

Furthermore, the number of $1$'s in column $j$ is equal to $\gamma^{\prime} + c^{\prime} + C^{\prime} + 1$, and by Proposition~\ref{number-omega-row-col} the number of $1$'s in column $j$ is equal to $e+\mu+1$.    Hence
\begin{equation} \label{EQ-6'}
\gamma^{\prime} + c^{\prime} + C^{\prime} = e+\mu.
\end{equation}

Finally, since $Q^2 = (n-1)I + \mu Q$ we have that
\begin{align*}
\mu  &= \mu Q_{ij} = [(n-1)I + \mu Q]_{ij} = Q^2_{ij} = \sum_{k=1}^n Q_{ik} Q_{kj} \\
&= \alpha^{\prime} (\omega \omega^2) + \beta^{\prime} (\omega \omega) + \gamma^{\prime} (\omega 1) + a^{\prime} (\omega^2 \omega^2) + b^{\prime} (\omega^2 \omega) + c^{\prime} (\omega^2 1) \\
& \qquad \qquad \qquad \qquad \qquad \qquad \qquad \qquad \qquad + A^{\prime} (1 \omega^2) + B^{\prime} (1 \omega) + C^{\prime} (1 \, 1) \\
&= \alpha^{\prime} 1  + \beta^{\prime} \omega^2 + \gamma^{\prime} \omega + a^{\prime} \omega + b^{\prime} 1 + c^{\prime} \omega^2 + A^{\prime} \omega^2 + B^{\prime} \omega + C^{\prime} 1 \\
&= (\alpha^{\prime}+ b^{\prime} + C^{\prime})1 + (\gamma^{\prime} + a^{\prime} + B^{\prime})\omega + (\beta^{\prime} + c^{\prime} + A^{\prime})\omega^2
\end{align*}
so that $$(\alpha^{\prime} + b^{\prime} + C^{\prime} - \mu)1 + (\gamma^{\prime} + a^{\prime} + B^{\prime})\omega + (\beta^{\prime} + c^{\prime} + A^{\prime})\omega^2 = 0.$$  It follows from Lemma~\ref{1-lemma} that $\alpha^{\prime}+ b^{\prime} + C^{\prime} - \mu = \gamma^{\prime} + a^{\prime} + B^{\prime} = \beta^{\prime} + c^{\prime} + A^{\prime}$ and thus
\begin{equation} \label{EQ-7'}
\alpha^{\prime} - \gamma^{\prime} -a^{\prime} + b^{\prime} -B^{\prime} + C^{\prime} = + \mu
\end{equation}
and 
\begin{equation} \label{EQ-8'}
\alpha^{\prime} - \beta^{\prime} + b^{\prime} -c^{\prime} -A^{\prime} +C^{\prime}= + \mu.
\end{equation}

When we row reduce this system we obtain the following equations
\begin{align}
\alpha^{\prime}= B^{\prime}  \tag{Eq.~8} \\
\beta^{\prime} = C^{\prime} - \mu \tag{Eq.~9} \\
\gamma^{\prime} = e + \mu - B^{\prime} - C^{\prime} \tag{Eq.~10} \\
a^{\prime} =  C^{\prime} - \mu \tag{Eq.~11} \\
b^{\prime} = e + \mu - B^{\prime} - C^{\prime} \tag{Eq.~12} \\
c^{\prime} = B^{\prime} \tag{Eq.~13} \\
A^{\prime} = e + \mu - B^{\prime} - C^{\prime} \tag{Eq.~14} 
\end{align}
with $B^{\prime}$ and $C^{\prime}$ as the two free variables.
Note that it follows from these equations that, $\alpha^{\prime} = c^{\prime} = B^{\prime}, \beta^{\prime} = a^{\prime} = C^{\prime} - \mu$ and $\gamma^{\prime} = A^{\prime} = b^{\prime} = e + \mu - B^{\prime} - C^{\prime}.$

The above equations are necessary and sufficient to characterize nontrivial cube root signature matrices of equiangular $(n,k)$-frames.

\begin{theorem} \label{nasc}  Let $Q$ be a self-adjoint $n \times n$
  matrix with $Q_{ii}=0, Q_{i1} = Q_{1i}=1$ for all $i$ and $Q_{ij}$ a
  cube root of unity for all $i \neq j, 2 \le i,j \le n$ with at least
  one entry that is not equal to $1.$  Then $Q^2 = (n-1)I + \mu Q$ if and only if for each pair $i \neq j$ such that $Q_{i,j}= \omega,$ conditions Eq.~1--Eq.~7 are satisfied and for each pair $i \neq j, 2 \le i,j \le n$ such that $Q_{i,j} =1,$ conditions Eq.~8--Eq.~14 are satisfied, where $3e+ \mu +2 =n.$
\end{theorem}
\begin{proof}  We have that $Q^2= (n-1)I + \mu Q$ if and only if for every $(i,j)$ the corresponding entries are equal. Equality of the $(i,i)$-th entries follows from the fact that the off-diagonal entries are  all of modulus one. 

The proof of Proposition~\ref{number-omega-row-col} shows that the relationships between $e, \mu$ and $n$ are equivalent to the condition that
the $(i,1)$-th entries and the $(1,i)$-th entries of $Q^2$ are equal to $\mu$ which are the corresponding $(i,1)$-entries and $(1,i)$-entries of $(n-1)I + \mu Q.$

The equations, Eq.~1--Eq.~7 are equivalent to requiring that if $Q_{i,j}= \omega,$ then the $(i,j)$-th entry of $Q^2$ is equal to $\mu \omega$ which is the $(i,j)$-th entry of $(n-1)I + \mu Q.$

If $Q_{i,j} = \omega^2,$ then $Q_{j,i} = \omega$ and by the last argument we have that the $(j,i)$-th entry of $Q^2$ is equal to the $(j,i)$-th entry of $(n-1)I + \mu Q.$  But since both of these matrices are self-adjoint, we have that equality of their $(j,i)$-th entry implies equality of their $(i,j)$-th entry.

Finally, if $Q_{i,j} =1, i \neq j, 2 \le i,j \le n,$ then the equations, Eq.~8--Eq.~14, are equivalent to the equality of the $(i,j)$-th entries of $Q^2$ and $(n-1)I + \mu Q.$
\end{proof}

\begin{remark}
It is important to realize that all of our results apply only to nontrivial cube root signature matrices, because
we assume that the first row contains all $1$'s except for one $0$, and the second row contains at least one $\omega$. 
For example, one can check that $Q = \left( \begin{smallmatrix} 0 & \omega & \omega^2 \\ \omega^2 & 0 & \omega \\ \omega & \omega^2 & 0 \end{smallmatrix} \right)$ is a $3 \times 3$ cube root signature matrix that satisfies $Q^2 = 2I + Q$.  However, conjugating $Q$ by the unitary matrix $\left( \begin{smallmatrix} 1 & 0 & 0 \\ 0 & \omega & 0 \\ 0 & 0 & \omega^2 \end{smallmatrix} \right)$ shows that $Q$ has standard form $\left( \begin{smallmatrix} 0 & 1 & 1 \\ 1 & 0 & 1 \\ 1 & 1 & 0 \end{smallmatrix} \right)$, and therefore $Q$ is trivial.  
Thus $e=0$ and the starting point (\ref{EQ-1}) for deriving our conditions does not  hold.
\end{remark}

Because $\alpha, \beta, \gamma$, $a, b, c$, $A, B, C$, $\alpha^{\prime}, \beta^{\prime}, \gamma^{\prime}$, $a^{\prime}, b^{\prime},c^{\prime}$, $A^{\prime}, B^{\prime}$ and $C^{\prime}$ are all non-negative integers, the  above equations have a number of consequences for the values $n$ and $\mu$.  We now state some results exploring these consequences. Remarkably, these consequences all  seem to stem from the unprimed equations.

\begin{proposition} \label{parameter-conds-prop}
Let $Q$ be a nontrivial cube root signature matrix of an equiangular $(n,k)$-frame, 
satisfying $Q^2 = (n-1)I + \mu Q$.  Then the following hold:
\begin{enumerate}
\item[(a)] The value $\mu$ is an integer and $\mu \equiv 1 \ (\textnormal{mod} \ 3)$.
\item[(b)] The integer $n$ satisfies $n \equiv 0 \ (\textnormal{mod} \ 3)$.
\item[(c)] If $\lambda_1 < 0 < \lambda_2$ are the eigenvalues of $Q$, then $\lambda_1$ and $\lambda_2$ are integers with $\lambda_1 \equiv 2 \ (\textnormal{mod} \ 3)$ and $\lambda_2 \equiv 2 \ (\textnormal{mod} \ 3)$.
\item[(d)] The integer $4(n-1) + \mu^2$ is a perfect square and in addition we have $4(n-1) + \mu^2 \equiv 0 \ (\textnormal{mod} \ 9)$.
\end{enumerate}
\end{proposition}

\begin{proof}
For (a) note that Eq.~7 shows that $A + B + C = (n + 2 \mu +1)/3$ and hence $A + B + C = e + \mu + 1$, with $e =  \frac{n-\mu-2}{3}$.  Since $e$ is an integer by Proposition~\ref{number-omega-row-col}, and since $A$, $B$, and $C$ are integers, it follows that $\mu$ is an integer.  In addition Eq.~1 shows that $2 \mu = -3 (\alpha - B) - 1$ so that $2 \mu \equiv 2 \ (\textnormal{mod} \ 3)$ and $\mu \equiv 1 \ (\textnormal{mod} \ 3)$.

For (b) we see that Eq.~3 implies $n = 3 (\gamma + B + C - \mu)$.  Since $\gamma$, $B$, and $C$ are integers, and since $\mu$ is an integer by (a), we have that $n  \equiv 0 \ (\textnormal{mod} \ 3)$.

For (c) we use \eqref{param-eqs} and (a) to see that $(n-2k) \sqrt{\frac{n-1}{k(n-k)}} = \mu \in \Z$, and thus $ \sqrt{\frac{n-1}{k(n-k)}}  \in \Q$ and $\lambda_1 = -\sqrt{\frac{k(n-1)}{(n-k)}} = - k^2 \sqrt{\frac{n-1}{k(n-k)}} \in \Q$.  In addition, using \eqref{param-eqs} we have that $\lambda_2 = (1-n)/\lambda_1 \in \Q$.  Thus $\lambda_1$ and $\lambda_2$ are both rational.  Because $Q$ satisfies $Q^2 = (n-1)I + \mu Q$, the polynomial $p(x) = x^2 - \mu x - (n-1)$ annihilates $Q$, and hence the minimal polynomial of $Q$ divides $p(x)$ and the eigenvalues $\lambda_1$ and $\lambda_2$ are rational roots of $p(x)$.  Since the coefficients of $p(x)$ are integers and the leading coefficient of $p(x)$ is $1$, the Rational Root Theorem tells us that the only rational roots of $p(x)$ are integers.  Hence $\lambda_1$ and $\lambda_2$ are integers.  Finally, using \eqref{param-eqs} we have that $\lambda_1 + \lambda_2 = \mu$, so by (a) we have 
\begin{equation} \label{lambda-eq-1}
\lambda_1 + \lambda_2 \equiv 1 \ (\textnormal{mod} \ 3).
\end{equation}
Also, using \eqref{param-eqs} we have that $\lambda_1 \lambda_2 = 1-n$, so by (b) we have 
\begin{equation} \label{lambda-eq-2}
\lambda_1 \lambda_2 \equiv 1 \ (\textnormal{mod} \ 3).
\end{equation}
There are only three possibilities for the residue of an integer modulo 3 (namely 0, 1, or 2) and a consideration of cases shows that the only situation in which \eqref{lambda-eq-1} and \eqref{lambda-eq-2} are both satisfied is when $\lambda_1 \equiv 2 \ (\textnormal{mod} \ 3)$ and $\lambda_2 \equiv 2 \ (\textnormal{mod} \ 3)$.

For (d) we use \eqref{param-eqs} to write $k = \frac{n}{2} - \frac{\mu n}{2 \sqrt{4(n-1) + \mu^2}}$.  Thus $ \sqrt{4(n-1) + \mu^2} = \frac{\mu n }{n-2k} \in \Q$ by (a).  Since $n$ and $\mu$ are integers, we have that $4(n-1) + \mu^2$ is an integer, and $\sqrt{4(n-1) + \mu^2}$ is rational if and only if $\sqrt{4(n-1) + \mu^2}$ is an integer.  Thus  $\sqrt{4(n-1) + \mu^2} = m$ for some $m \in \Z$ and $4(n-1) + \mu^2 = m^2$, so that $4(n-1) + \mu^2$ is a perfect square.  Furthermore, since $4(n-1) + \mu^2 = m^2$ and we have $\mu \equiv 1 \ (\textnormal{mod} \ 3)$ by (a) and $n \equiv 0 \ (\textnormal{mod} \ 3)$ by (b), it follows that $m^2 \equiv 0 \ (\textnormal{mod} \ 3)$.  Thus $3$ divides $m^2$, and since $3$ is prime, we have that $3$ divides $m$.  Hence $9$ divides $m^2 = 4(n-1) + \mu^2$ and $4(n-1) + \mu^2 \equiv 0 \ (\textnormal{mod} \ 9)$.
\end{proof}

Proposition~\ref{parameter-conds-prop} shows that $\mu \equiv 1 \ (\textnormal{mod} \ 3)$ and $n \equiv 0 \ (\textnormal{mod} \ 3)$.  However, we can do slightly better than this, as the following proposition shows.

\begin{proposition} \label{mod9-prop}
Let $Q$ be a nontrivial cube root signature matrix of an equiangular $(n,k)$-frame, satisfying $Q^2 = (n-1)I + \mu Q$.  Then one of the following three cases must hold:
\begin{itemize}
\item $n \equiv 0 \ (\textnormal{mod} \ 9)$ and $\mu \equiv 7 \ (\textnormal{mod} \ 9)$, or 
\item $n \equiv 3 \ (\textnormal{mod} \ 9)$ and $\mu \equiv 1 \ (\textnormal{mod} \ 9)$, or 
\item $n \equiv 6 \ (\textnormal{mod} \ 9)$ and $\mu \equiv 4 \ (\textnormal{mod} \ 9)$.
\end{itemize}
\end{proposition}

\begin{proof}
From Proposition~\ref{parameter-conds-prop}(b), we have that $n \equiv 0 \ (\textnormal{mod} \ 3)$ so that $n$ is congruent to $0$, $3$, or $6$ modulo $9$.  Also, from Proposition~\ref{parameter-conds-prop}(a), we have that $\mu \equiv 0 \ (\textnormal{mod} \ 3)$ so that $\mu$ is congruent to $1$, $4$, or $7$ modulo $9$.  By Proposition~\ref{parameter-conds-prop}(d) we have $4(n-1) + \mu^2 \equiv 0 \ (\textnormal{mod} \ 9)$ and by considering the possible values of $n$ and $\mu$ modulo $9$, we see the only way this equation is satisfied is if one of the three cases in the statement of this proposition holds.
\end{proof}

\begin{proposition} \label{e-mod3-prop}
Let $Q$ be a nontrivial cube root signature matrix of an equiangular $(n,k)$-frame, satisfying $Q^2 = (n-1)I + \mu Q$.  If we set $e := \frac{n- \mu -2}{3}$, then $e$ is an integer with $$e \equiv 0 \ (\textnormal{mod} \ 3)$$  and $e$ satisfies $$\frac{2n}{9} \leq e \leq \frac{4n-9}{9}.$$
\end{proposition}

\begin{proof}
The fact that $e$ is an integer follows from Proposition~\ref{number-omega-row-col}.  Also since $3e = n - \mu - 2$, by considering the three possibilities of Proposition~\ref{mod9-prop}, we see that in any of these three cases we have $3e \equiv 0 \ (\textnormal{mod} \ 9)$.  Hence $e \equiv 0 \ (\textnormal{mod} \ 3)$.

Next we look at Eq.~1, which shows that $\alpha - B = -\frac{2\mu}{3} - \frac{1}{3}$.  Since $\alpha$ is a non-negative integer we have that $0 \leq \alpha = B -\frac{2\mu}{3} - \frac{1}{3}$ and 
\begin{equation} \label{B-ineq}
B \geq \frac{2 \mu +1}{3}.
\end{equation}
Also, by Eq.~2 we have that $\beta - C = -\frac{2\mu}{3} - \frac{4}{3}$.  Since $\beta$ is a non-negative integer, we have that $0 \leq \beta = C -\frac{2\mu}{3} - \frac{4}{3}$ and 
\begin{equation} \label{C-ineq}
C \geq \frac{2 \mu +4}{3}.
\end{equation}
In addition, by Eq.~5 we have that $b + B + C = \frac{n}{3} + \frac{\mu}{3} - \frac{1}{3} = e + \frac{2\mu +1}{3}$.  Since $b$ is a non-negative integer, we may use \eqref{B-ineq} and \eqref{C-ineq} to obtain $$e +  \frac{2\mu +1}{3} = b + B + C \geq B + C \geq \frac{2 \mu +1}{3} + \frac{2 \mu +4}{3}$$ so that $e \geq  \frac{2\mu + 4}{3}$.  Thus $3e \geq 2 \mu + 4 = 2n - 6e$, and $$e \geq \frac{2n}{9}.$$  

For the upper bound on $e$, we may assume that $Q$ is in standard form, so that every row of $Q$ contains at least one $1$, and consequently $C \geq 1$.  Using Eq.~3 and the fact that $\gamma$ and $B$ are non-negative integers, we have that $$1 \leq \gamma + B + C = \frac{n}{3} + \mu = -3e + \frac{4n}{3} - 2$$ so $3e \leq \frac{4n}{3} - 3$ and $$e \leq \frac{4n-9}{9}.$$
\end{proof}

\begin{remark} By Proposition~\ref{number-omega-row-col}, there are $e$ entries that are equal to $\omega$ and $e$ entries that are equal to $\omega^2.$  Thus, excluding the $Q_{i,1}$ entry, there are $n-2-2e$ entries that are equal to 1.
Thus, by the above inequalities, for each $i, 2 \le i \le n,$ there are always at least $n-2-2e \ge n -2 - 2\frac{4n-9}{9} = \frac{n}{9}$ values of $j, 2 \le j \le n$ for which $Q_{i,j}= 1.$ Hence, there is always at least one such entry.
\end{remark}

\begin{corollary} \label{mu-mod3-cor}
Let $Q$ be a nontrivial cube root signature matrix of an equiangular $(n,k)$-frame satisfying $Q^2 = (n-1)I + \mu Q$.  If we set $e := \frac{n- \mu -2}{3}$, then $\mu$ is an integer with $$\mu \equiv 1 \ (\textnormal{mod} \ 3)$$  and $\mu$ satisfies $$1 - \frac{n}{3} \leq \mu \leq \frac{n}{3}-2.$$
\end{corollary}

\begin{remark}
The fact that $\mu$ is an integer congruent to $1$ modulo $3$ is shown in Proposition~\ref{parameter-conds-prop}.  Using that $e = \frac{n-\mu-2}{3}$ we may rewrite the inequality in Proposition~\ref{e-mod3-prop} in terms of $\mu$ to obtain $1 - \frac{n}{3} \leq \mu \leq \frac{n}{3}-2$.
\end{remark}

\section{Narrowing the search for cube root signature matrices} \label{algorithm-sec}

In Section~\ref{Eqs-sec} we derived a number of conditions that the parameters (e.g., $n,k,\mu,e$) of a nontrivial cube root Seidel matrix must satisfy to make it the signature matrix of an equiangular $(n,k)$-frame.  These conditions allow us to rule out a number of possible values for $n$ and $k$.

In particular, the previous results can be incorporated in an algorithm 
to determine the possible $k$ values for a given $n$.

$ $

\noindent \fbox{ \parbox{4.83in}{ 

\begin{center} \underline{\textbf{Algorithm for deducing possible $\boldsymbol{(n,k)}$ values}} \end{center}

\smallskip

\noindent Begin with a value of $n$ that is divisible by $3$ (see Proposition~\ref{parameter-conds-prop}(b)).

\smallskip

\smallskip

\noindent \textbf{Step 1:}  Find all values of $e$ satisfying $\frac{2n}{9} \leq e \leq \frac{4n-9}{9}$ with $e  \equiv 1 \ (\textnormal{mod} \ 3)$.  (See Proposition~\ref{e-mod3-prop}.)

\smallskip

\smallskip

\noindent \textbf{Step 2:}  For each $e$ from Step 1, calculate the value of $\mu = n-3e-2$.

\smallskip

\smallskip

\noindent \textbf{Step 3:}  For each $\mu$ from Step 2, calculate the value of $k = \frac{n}{2} - \frac{\mu n}{2 \sqrt{4(n-1) + \mu^2}}$ (see \eqref{param-eqs}).  The only allowable $(n,k)$-frames with nontrivial 
signature matrices are those with $k$ equal to an integer greater than one.  A necessary condition for $k$ to be an integer is that $\sqrt{4(n-1) + \mu^2}$ is rational.   }}

$ $

$ $

\begin{remark}
One may wonder why in our algorithm we do not simply use Corollary~\ref{mu-mod3-cor} to first find the values of $\mu$ satisfying $1 - \frac{n}{3} \leq \mu \leq \frac{n}{3}-2$ with $\mu \equiv 1 \ (\textnormal{mod} \ 3)$, and then proceed directly to Step 3.  This would seem to eliminate the need to calculate the value of $e$, and reduce our algorithm by one step.  It turns out, however, that this is less efficient.  Since $e = \frac{n-\mu-2}{3}$, there will in general be less values of $e$ found in Step 1 than there will be values of $\mu$ satisfying the condition of Corollary~\ref{mu-mod3-cor}.
\end{remark}

To demonstrate our algorithm we go through the calculations of possible values of $(n,k)$ for $2\le k<n \leq 50$.  To do this we use our algorithm on all values of $n$ that are less than $50$ and a multiple of $3$.

$ $

\noindent \underline{n=3}.  Step 1: We need $\frac{2}{3} \leq e$ and $e \leq \frac{1}{3}$, which cannot occur.  Thus there are no allowable values of $k$ in this case which lead to a nontrivial signature matrix.

\smallskip

\noindent \underline{n=6}.  Step 1: We need $1\frac{1}{3} \leq e \leq 1\frac{2}{3}$ and $e \equiv 0 \ (\textnormal{mod} \ 3)$, which cannot occur. Thus there are no allowable values of $k$ in this case.

\smallskip

\noindent \underline{n=9}.  Step 1: We need $2 \leq e \leq 3$ and $e  \equiv 0 \ (\textnormal{mod} \ 3)$.  Thus $e$ equals $3$.  Step 2: For $e=3$ we have $\mu = -2$.  Step 3: For $\mu = -2$ we have $k = \frac{n}{2} - \frac{\mu n}{2 \sqrt{4(n-1) + \mu^2}} = 6$.  Thus $(9,6)$ is an allowable $(n,k)$ value.

\smallskip

\noindent \underline{n=12}.  Step 1: We need $2\frac{2}{3} \leq e \leq 4\frac{1}{3}$ and $e  \equiv 0 \ (\textnormal{mod} \ 3)$.  Thus $e$ equals $3$.  Step 2: For $e=3$ we have $\mu = 1$.  Step 3: For $\mu = 1$ we have $\sqrt{ 4(n-1) + \mu^2} = \sqrt{45} \notin \Q$, which cannot occur.  Thus there are no allowable values of $k$ in this case.

\smallskip

\noindent \underline{n=15}.  Step 1: We need $3\frac{1}{3} \leq e \leq 5\frac{2}{3}$ and $e  \equiv 0 \ (\textnormal{mod} \ 3)$, which cannot occur.  Thus there are no allowable values of $k$ in this case.

\smallskip

\noindent \underline{n=18}.  Step 1: We need $4 \leq e \leq 7$ and $e  \equiv 0 \ (\textnormal{mod} \ 3)$.  Thus $e$ equals $6$.  Step 2: For $e=6$ we have $\mu = -2$.  Step 3: For $\mu = -2$ we have $\sqrt{ 4(n-1) + \mu^2} = \sqrt{72} \notin \Q$, which cannot occur.  Thus there are no allowable values of $k$ in this case.

\smallskip

\noindent \underline{n=21}.  Step 1: We need $4\frac{2}{3} \leq e \leq 8\frac{1}{3}$ and $e  \equiv 0 \ (\textnormal{mod} \ 3)$.  Thus $e$ equals $6$.  Step 2: For $e=6$ we have $\mu = 1$.  Step 3: For $\mu = 1$ we have $k = \frac{n}{2} - \frac{\mu n}{2 \sqrt{4(n-1) + \mu^2}} = \frac{28}{3} \notin \Q$.  Thus there are no allowable values of $k$ in this case.

\smallskip

\noindent \underline{n=24}.  Step 1: We need $5\frac{1}{3} \leq e \leq 9\frac{2}{3}$ and $e  \equiv 0 \ (\textnormal{mod} \ 3)$.  Thus $e$ equals $6$ or $9$.  Step 2: For $e=6$ we have $\mu = 4$, and for $e=9$ we have $\mu = -5$.  Step 3: For $\mu = 4$ we have $\sqrt{ 4(n-1) + \mu^2} = \sqrt{108} \notin \Q$, which cannot occur.  For $\mu = -5$ we have $\sqrt{ 4(n-1) + \mu^2} = \sqrt{117} \notin \Q$, which cannot occur.  Thus there are no allowable values of $k$ in this case.

\smallskip

\noindent \underline{n=27}.  Step 1: We need $6 \leq e \leq 11$ and $e  \equiv 0 \ (\textnormal{mod} \ 3)$.  Thus $e$ equals $6$ or $9$.  Step 2: For $e=6$ we have $\mu = 7$, and for $e=9$ we have $\mu = -2$.  Step 3: For $\mu = 7$ we have $\sqrt{ 4(n-1) + \mu^2} = \sqrt{153} \notin \Q$, which cannot occur.  For $\mu = -2$ we have $\sqrt{ 4(n-1) + \mu^2} = \sqrt{108} \notin \Q$, which cannot occur.  Thus there are no allowable values of $k$ in this case.

\smallskip

\noindent \underline{n=30}.  Step 1: We need $6\frac{2}{3} \leq e \leq 12\frac{1}{3}$ and $e  \equiv 0 \ (\textnormal{mod} \ 3)$.  Thus $e$ equals $9$ or $12$.  Step 2: For $e=9$ we have $\mu = 1$, and for $e=12$ we have $\mu = -8$.  Step 3: For $\mu = 1$ we have $\sqrt{ 4(n-1) + \mu^2} = \sqrt{117} \notin \Q$, which cannot occur.  For $\mu = -8$ we have $\sqrt{ 4(n-1) + \mu^2} = \sqrt{180} \notin \Q$, which cannot occur.  Thus there are no allowable values of $k$ in this case.

\smallskip

\noindent \underline{n=33}.  Step 1: We need $7\frac{1}{3} \leq e \leq 13\frac{2}{3}$ and $e  \equiv 0 \ (\textnormal{mod} \ 3)$.  Thus $e$ equals $9$ or $12$.  Step 2: For $e=9$ we have $\mu = 4$, and for $e=12$ we have $\mu = -5$.  Step 3: For $\mu = 4$ we have $k = \frac{n}{2} - \frac{\mu n}{2 \sqrt{4(n-1) + \mu^2}} = 11$, which is allowed.   For $\mu = -5$ we have $\sqrt{ 4(n-1) + \mu^2} = \sqrt{153} \notin \Q$, which cannot occur.  Thus $(33,11)$ is the only allowable $(n,k)$ value for $n=33$.

\smallskip

\noindent \underline{n=36}.  Step 1: We need $8 \leq e \leq 15$ and $e \equiv 0 \ (\textnormal{mod} \ 3)$.  Thus $e$ equals $9$ or $12$ or $15$.  Step 2: For $e=9$ we have $\mu = 7$, for $e=12$ we have $\mu = -2$, and for $e=15$ we have $\mu = -11$.  Step 3: For $\mu = 7$ we have $\sqrt{ 4(n-1) + \mu^2} = \sqrt{189} \notin \Q$, which cannot occur. For $\mu = -2$ we have $k = \frac{n}{2} - \frac{\mu n}{2 \sqrt{4(n-1) + \mu^2}} = 21$, which is allowed.   For $\mu = -11$ we have $\sqrt{ 4(n-1) + \mu^2} = \sqrt{261} \notin \Q$, which cannot occur.  Thus $(36,21)$ is the only allowable $(n,k)$ value for $n=36$.

\smallskip

\noindent \underline{n=39}.  Step 1: We need $8\frac{2}{3} \leq e \leq 16\frac{1}{3}$ and $e \equiv 0 \ (\textnormal{mod} \ 3)$.  Thus $e$ equals $9$ or $12$ or $15$.  Step 2: For $e=9$ we have $\mu = 10$, for $e=12$ we have $\mu = 1$, and for $e=15$ we have $\mu = -8$.  Step 3: For $\mu = 10$ we have $\sqrt{ 4(n-1) + \mu^2} = \sqrt{251} \notin \Q$, which cannot occur. For $\mu = 1$ we have $\sqrt{ 4(n-1) + \mu^2} = \sqrt{153} \notin \Q$, which cannot occur.   For $\mu = -8$ we have $\sqrt{ 4(n-1) + \mu^2} = \sqrt{216} \notin \Q$, which cannot occur.  Thus there are no allowable values of $k$ in this case.

\smallskip

\noindent \underline{n=42}.  Step 1: We need $9\frac{1}{3} \leq e \leq 17\frac{2}{3}$ and $e \equiv 0 \ (\textnormal{mod} \ 3)$.  Thus $e$ equals $12$ or $15$.  Step 2: For $e=12$ we have $\mu = 4$, and for $e=15$ we have $\mu = -5$.  Step 3: For $\mu = 4$ we have $\sqrt{ 4(n-1) + \mu^2} = \sqrt{180} \notin \Q$, which cannot occur. For $\mu = -5$ we have $\sqrt{ 4(n-1) + \mu^2} = \sqrt{189} \notin \Q$, which cannot occur.  Thus there are no allowable values of $k$ in this case.

\smallskip

\noindent \underline{n=45}.  Step 1: We need $10 \leq e \leq 19$ and $e \equiv 0 \ (\textnormal{mod} \ 3)$.  Thus $e$ equals $12$ or $15$ or $18$.  Step 2: For $e=12$ we have $\mu = 7$, for $e=15$ we have $\mu = -2$, and for $e=18$ we have $\mu = -11$.  Step 3: For $\mu = 7$ we have $k = \frac{n}{2} - \frac{\mu n}{2 \sqrt{4(n-1) + \mu^2}} = 12$, which is allowed.  For $\mu = -2$ we have $\sqrt{ 4(n-1) + \mu^2} = \sqrt{180} \notin \Q$, which cannot occur.  For $\mu = -11$ we have $\sqrt{ 4(n-1) + \mu^2} = \sqrt{297} \notin \Q$, which cannot occur.  Thus $(45,12)$ is the only allowable $(n,k)$ value for $n=45$.

\smallskip

\noindent \underline{n=48}.  Step 1: We need $10\frac{2}{3} \leq e \leq 20\frac{1}{3}$ and $e \equiv 0 \ (\textnormal{mod} \ 3)$.  Thus $e$ equals $12$ or $15$ or $18$.  Step 2: For $e=12$ we have $\mu = 10$, for $e=15$ we have $\mu = 1$, and for $e=18$ we have $\mu = -8$.  Step 3: For $\mu = 10$ we have $\sqrt{ 4(n-1) + \mu^2} = \sqrt{288} \notin \Q$, which cannot occur. For $\mu = 1$ we have $\sqrt{ 4(n-1) + \mu^2} = \sqrt{189} \notin \Q$, which cannot occur.   For $\mu = -8$ we have $\sqrt{ 4(n-1) + \mu^2} = \sqrt{252} \notin \Q$, which cannot occur.  Thus there are no allowable values of $k$ in this case.

We list the possible values of $(n,k)$ in Table~1.  In addition, although we do not reproduce the calculations here, the authors have computed the possible $(n,k)$ values from $50$ to $100$ as well, and these are also listed in Table~1.

\section{Graphs of cube root signature matrices}

Arguably, the most successful means to find equiangular tight frames in the real case has been via the correspondence 
between graphs and signature matrices.
Much of the graph-theoretic approach to real signature matrices, due largely to \cite{LS73}, can be
repeated in our case by replacing graphs by {\em directed graphs.} 
In this section we make the connection between cube root signature matrices and directed graphs 
explicit and describe necessary and sufficient conditions that a directed graph must satisfy in order to 
give rise to a cube root signature matrix in standard form.

We define a one-to-one correspondence between $n \times n$ selfadjoint
matrices whose diagonal entries are zero and whose off-diagonal
entries are cube roots of unity with directed graphs (with no loops) on $n$ vertices 
in the following manner.
Given such a matrix $Q$, we associate
its rows and columns with vertices, and an entry $Q_{i,j}=\omega$
with a directed edge from the $i$th to the $j$th vertex, while
$Q_{i,j}=\omega^2$ is associated with a directed edge from the $j$th to the
$i$th vertex. When $Q_{i,j} =1 =Q_{j,i}$, then there is no edge
connecting the $i$th and $j$th vertex. We denote this directed graph
by $G(Q).$
Conversely, given any directed graph $G$ on $n$ vertices, and an
enumeration of the vertices, we obtain such a matrix which we denote
by $Q_G$ and we call this matrix the {\em Seidel adjacency matrix} of the
directed graph.

Given a directed graph $G$ and vertices $v$ and $w$ we write $v \to w$
or $w \leftarrow v$
to indicate that there is a directed edge from $v$ to $w$.
We say that $v$ {\em emits} the edge and that $w$ {\em receives} the edge. 
When there is no directed edge (in either direction) between $v$ and
$w$, we write $v \wr w.$

A key element of Seidel's theory \cite{LS73} was the
introduction of his switching equivalence of graphs. 
We reformulate the equivalence relation
for complex Seidel matrices for directed graphs.

Suppose that we have a self-adjoint matrix $Q$ whose off-diagonal entries are cube roots of unity and we associate a directed graph to $Q$ as above.  If we let $D$ be the diagonal matrix satisfying $d_{j,j} =1, j \ne i$ and $d_{i,i} = \omega$, then the directed graph $G_1$ associated to $DQD^*$ is obtained from the directed graph $G$ associated to $Q$ by
\begin{itemize}
\item inserting a directed edge from $i$ to $j$ whenever $i$ and $j$ had no edge, 
\item replacing a directed edge from $i$ to $j$ by a directed edge from $j$ to $i$, and
\item deleting all directed edges from $j$ to $i$.
\end{itemize}

We shall refer to this set of operations as the {\em $\omega$-switching on the $i$th vertex.} Note that unlike Seidel's switching for undirected graphs\cite{LS73}, when we switch twice on the $i$th vertex, we do not return to the original graph. Here when we switch three times, we return to the original directed graph.
If we consider cube root signature matrices that are in standard form,
then in the corresponding directed graph the first vertex neither
emits nor receives any edges.
Thus, without loss of generality, we may ignore this vertex and focus
on the directed subgraph on $m=n-1$ vertices. 

We now wish to  describe explicitly in graph theoretical terms the
directed graphs that correspond to cube root signature matrices in
standard form.

\begin{definition} We call a directed graph on $m$ vertices {\em $e$-regular,} provided that each vertex emits exactly $e$ directed edges and receives exactly $e$ directed edges.
\end{definition}

Given a directed graph $G$ and vertices $v$ and $w$ we now wish to
describe the parameters that are the natural equivalents of the
parameters of Section~3.

Given a directed graph $G$ with vertex set $V$ and vertices $v$ and $w$ with $v \to w,$ we
set
\begin{align*}
\alpha &= \# \{ u \in V : v \to u \text{ and } w \to u \} \\
\beta &= \# \{ u \in V : v \to u \text{ and } u \to w \} \\
\gamma &= \# \{ u \in V : v \to u \text{ and } u \wr w \} \\
a &= \# \{ u \in V : u \to v \text{ and }  w \to u \} \\
b &= \# \{ u \in V : u \to v \text{ and } u \to w \} \\
c &= \# \{ u \in V : u \to v \text{ and } u \wr w \} \\
A &= \# \{ u \in V : u \wr v \text{ and } w \to u \} \\
B &= \# \{ u \in V : u \wr v \text{ and } u \to w \} \\
C_1 &= \# \{ u \in V : u \wr v \text{ and } u \wr w \}.
\end{align*}

We have introduced the slight change of notation, $C_1,$ from our
earlier notation $C,$ because since we have omitted the corresponding
first vertex, we will have that $C_1 = C -1.$ 

Similarly, given a directed graph $G$ with vertex set $V$ and vertices
$v$ and $w$ such that $v \wr w,$ we set

\begin{align*}
\alpha^{\prime} &= \# \{ u \in V : v \to u \text{ and } w \to u \} \\
\beta^{\prime} &= \# \{ u \in V : v \to u \text{ and } u \to w \} \\
\gamma^{\prime} &= \# \{ u \in V : v \to u \text{ and } u \wr w \} \\
a^{\prime} &= \# \{ u \in V : u \to v \text{ and } w \to u \} \\
b^{\prime} &= \# \{ u \in V : u \to v \text{ and } u \to w \} \\
c^{\prime} &= \# \{ u \in V : u \to v \text{ and } u \wr w \} \\
A^{\prime} &= \# \{ u \in V : u \wr v \text{ and } w \to u \} \\
B^{\prime} &= \# \{ u \in V : u \wr v \text{ and } u \to w \} \\
C_1^{\prime} &= \# \{ u \in V : u \wr v \text{ and } u \wr w \}.
\end{align*}
Again we will have that $C_1^{\prime} = C^{\prime} -1.$

We now reinterpret Theorem~\ref{nasc} in the terminology of directed graphs.

\begin{theorem} \label{nasc-digraph} There exists a nontrivial $n \times n$
  matrix $Q$ in standard form whose off-diagonal entries are cube
  roots of unity and such that $Q$ satisfies $Q^2 = (n-1)I + \mu Q,$
  for some $\mu,$
 if and only if for some $e > 0$ there exists a non-empty $e$-regular directed graph on $m=n-1$ vertices such that each pair of vertices with a directed edge from $v$ to $w$ satisfies,
\begin{align}
\alpha - B &= \frac{6e+1-2m}{3} \tag{Eq.~d.1} \\
\beta - C_1 &= \frac{6e-2m+1}{3} \tag{Eq.~d.2} \\
\gamma + B + C_1 &= \frac{4m-9e-5}{3} \tag{Eq.~d.3} \\
 a - C_1 &= \frac{3e-m+2}{3}  \tag{Eq.~d.4} \\
 b + B + C_1 &= \frac{2m-3e-4}{3} \tag{Eq.~d.5} \\
 c - B &= \frac{3e-m+2}{3} \tag{Eq.~d.6} \\
 A + B + C_1 &=  m-2e-1, \tag{Eq.~d.7}  
\end{align}
while each pair of vertices $v$ and $w$ with no edges between them satisfies
\begin{align}
\alpha^{\prime}= B^{\prime}  \tag{Eq.~d.8} \\
\beta^{\prime} = C^{\prime}_1 +3e+2 -m \tag{Eq.~9} \\
\gamma^{\prime} = m -2e -2 - B^{\prime} - C^{\prime}_1 \tag{Eq.~d.10} \\
a^{\prime} =  C^{\prime}_1 + 3e + 2 -m \tag{Eq.~d.11} \\
b^{\prime} = m -2e -2 - B^{\prime} - C^{\prime}_1 \tag{Eq.~d.12} \\
c^{\prime} = B^{\prime} \tag{Eq.~d.13} \\
A^{\prime} = m - 2e - 2 - B^{\prime} - C^{\prime}_1. \tag{Eq.~d.14} 
\end{align}
\end{theorem}

Moreover, in the above case we have that $\mu = m - 3e -1$ and we obtain $Q$ from the directed graph on $n-1$ vertices by adjoining a row and column of $+1$'s to the adjacency matrix of the directed graph.

The results of Section~3 can now be interpreted as necessary conditions on $m$
and $e$ for the existence of such graphs.

If $G$ is the directed graph associated to $Q$ and after a sequence of switchings on various vertices, we obtain a new directed graph $G_1$, then it is easy to see that the matrix $Q_1$ associated to $G_1$ satisfies $Q_1 = DQD^*$ for some diagonal matrix whose entries are cube roots of unity. In fact $D$ can be taken to be the diagonal matrix whose $i$th diagonal entry is $\omega^{s_i}$ where $s_i$ is the number of times that we switched on the $i$th vertex.

The proof of Lemma~\ref{permutation-lemma} shows that given a directed
graph $G$ on $n$ vertices and a fixed vertex $v$, then $G$
is switching equivalent to a unique directed graph where the vertex
$v$ is isolated, i.e., such that $v$ neither emits nor receives any
directed edges. We let $G_v$ denote the directed graph on $n-1$
vertices that one obtains by deleting this isolated vertex.

\begin{proposition}\label{switching} Let $G$ be a directed graph on $n$ vertices with
  adjacency matrix $Q$ Then the following are equivalent:
\begin{enumerate}
\item $Q^2 = (n-1)I + \mu Q$ for some real number $\mu,$ 
\item for some vertex $v$, the graph $G_v$ is $e$-regular and
  satisfies Eq.~d.1--Eq.~d.14,
\item for every vertex $v$, the graph $G_v$ is $e$-regular and
  satisfies Eq.~d.1--Eq.d.14.
\end{enumerate}
Moreover, in this case $\mu = n -3e -2.$
\end{proposition}
\begin{proof} We have that $Q^2 = (n-1)I + \mu Q$ if and only if
  $(D^*QD)^2 = (n-1)I + \mu (D^*QD)$ for any diagonal matrix of whose
  entries consist of cube roots of unity. These matrices represent the
  adjacency matrices of all directed graphs switching equivalent to
  $G$ with the same enumeration of vertices.
Thus, if the adjacency matrix for $G$ satisfies the above equation,
then the adjacency matrix for every directed graph switching
equivalent to $G$
satisfies the above equation. The result now follows by applying
theorem~\ref{nasc-digraph}.
\end{proof}

We conclude that a necessary condition for the existence of cube root
signature matrices is the existence of directed graphs without undirected edges for which all vertices have 
the same in-degree and out-degree $e$. Moreover, if there is an edge $v\to w$, then
the paths between $v$ to $w$ of length 2
satisfy conditions (Eq.~d.1-7) and if there is no edge between $v$ and $w$, then
they satisfy conditions (Eq.~d.8-14). If there is only one solution
for these parameters, then this implies that the graph is
a directed strongly regular graph  \cite{Duval88}.

\section{Examples of cube root signature matrices with two eigenvalues}

\subsection{A first example}

In the literature on directed strongly regular  graphs, there is a graph on 8 vertices,
all with same in-degree and out-degree $e=3$ \cite[Example 4.1]{KMMZ97}. Indeed, it turns
out that the cube root Seidel matrix of this directed graph is the signature matrix
of an equiangular $(9,6)$-frame. This is, in fact, the unique frame up to switching, 
which follows from the uniqueness of directed strongly regular graphs on 8
vertices with in-degree and out-degree 3 \cite{Ham83}.

\begin{theorem} \label{9x9-sig-thm}
The matrix $$Q = \begin{pmatrix}
0 & 1 & 1 & 1 & 1 & 1 & 1 & 1 & 1 \\
1 & 0 & 1 & \omega & \omega & \omega & \omega^2 & \omega^2 & \omega^2 \\
1 & 1 & 0 & \omega^2 & \omega^2 & \omega^2 & \omega & \omega & \omega \\
1 & \omega^2 & \omega & 0 & \omega & \omega^2 & 1 & \omega & \omega^2 \\
1 & \omega^2 & \omega & \omega^2 & 0 & \omega & \omega & \omega^2 & 1 \\
1 & \omega^2 & \omega & \omega & \omega^2 & 0 & \omega^2 & 1 & \omega \\
1 & \omega & \omega^2 & 1 & \omega^2 & \omega & 0 & \omega^2 & \omega \\
1 & \omega & \omega^2 & \omega^2 & \omega & 1 & \omega & 0 & \omega^2 \\
1 & \omega & \omega^2 & \omega & 1 & \omega^2 & \omega^2 & \omega & 0
\end{pmatrix}$$ is a $9 \times 9$ nontrivial cube root signature
matrix of an equiangular $(9,6)$-frame.  Furthermore, any $9 \times 9$ nontrivial 
cube root signature matrix belonging to $k=6$ is switching equivalent to $Q$.
\end{theorem}

\begin{proof}

To see that $Q$ is a signature matrix, one need only verify that $Q^2 = 8I -2Q$.

We prove the claimed uniqueness using
a graph-theoretic argument \cite{Ham83}.  It will be enough to show that up to a
renumbering of the vertices there is a unique directed graph on 8
vertices satisfying the conditions, Eq.~d.1--Eq.~d.14.
To this end first note that since $e=3,$ there will be exactly $9-2e-2
=1$ entry in each row, other than $Q_{i,1}$, that is equal to $1.$
Thus, in the directed graph on 8 vertices, each vertex will emit 3
directed edges, receive 3 directed edges, and not be connected
 to exactly one vertex.

However, there is a unique 6-regular graph on 8 vertices.
This is because each non-adjacent pair of vertices has
6 common neighbors. Thus, the graph can be partitioned
into 4 such pairs, and is seen to be the complete 4-partite graph
corresponding to this partition of 8 vertices in sets of size 2.

Next one finds that the equations Eq.~d.1--Eq.~d.14 have a unique set
of solutions, $\alpha = \beta = a = b = c = A =1, \gamma =B = C_1 =0$
and $\alpha^{\prime} =\gamma^{\prime} = b^{\prime} = c^{\prime} =
A^{\prime} = 0, \beta^{\prime} = a^{\prime} =3.$ Thus, in the directed
graph, if $v \wr w,$ then since $\beta^{\prime} = 3,$ the three
directed edges emitted by $v$ terminate at the three directed edges
received by $w$. Similarly, since $a^{\prime} =3,$ the three directed
edges received by $v$ are emitted by the three vertices where the
three directed edges emitted by $w$ terminate. 

One now finds that there is a unique directed graph on 8 vertices
satisfying these relations.
\end{proof}

\subsection{Creating new signature matrices from old}

In this section we discuss ways to form new signature matrices from existing ones.

\begin{proposition} \label{bootstrap-prop}
Suppose that $Q_1$ and $Q_2$ are $n_1 \times n_1$ and $n_2 \times n_2$ signature matrices 
satisfying the 
equation $$Q_i^2 = (n_i-1)I_{n_i} -2 Q_i,\qquad i \in \{1,2\}$$  
Then  the matrix $Q := (Q_1+I_{n_1}) \otimes (Q_2+I_{n_2}) - I_{n_1n_2}$ is an $n_1 n_2\times n_1 n_2$ signature matrix satisfying the equation $$Q^2 = (n_1 n_2-1)I_{n_1 n_2} - 2 Q.$$
\end{proposition}

\begin{proof}
We first observe that $Q$ is self-adjoint, and since $Q_1$ and $Q_2$ contain unimodular off-diagonal entries 
and the identity matrix $I$ contains $1$'s along the diagonal, we have that the diagonal entries of $Q$ are zero, and the off-diagonal entries of $Q$ are unimodular.  Furthermore,
\begin{align*}
(Q+I_{n_1 n_2})^2 &= [(Q_1+I_{n_1}) \otimes  (Q_2+I_{n_2})]^2 \\
&= (Q_1+I_{n_1})^2 \otimes (Q_2+I_{n_2})^2 \\
&= (Q_1^2 + 2Q_1 + I_{n_1}) \otimes (Q_2^2 + 2Q_2 + I_{n_2}) \\
&= ((n_1-1)I_{n_1} - 2Q_1 + 2Q_1 + I_{n_1}) \otimes  ((n_2-1)I_{n_2} - 2Q_2 + 2Q_2 + I_{n_2}) \\
&= n_1 I_{n_1} \otimes  n_2 I_{n_2} \\
&= n_1 n_2  (I_{n_1} \otimes  I_{n_2}) \\
&= n_1 n_2  I_{n_1 n_2}.
\end{align*}
Thus $Q^2 + 2Q + I_{n_1 n_2} = n_1 n_2 I_{n_1 n_2}$, and $Q^2 = (n_1 n_2 -1)I_{n_1 n_2} - 2Q$, 
so $Q$ is a signature matrix.
\end{proof}

\begin{theorem} \label{bootstrap-9-thm}
For each $m \in \N$ there exists a nontrivial $9^m \times 9^m$ cube root signature matrix $Q$ satisfying $Q^2 = (9^m-1)I_{9^m} - 2 Q$.
\end{theorem}

\begin{proof}
When $m=1$ it follows from Theorem~\ref{9x9-sig-thm} that there exists a $9 \times 9$ nontrivial cube root signature matrix $Q'$ satisfying $(Q')^2 = 8I_9 -2Q'$.
To obtain matrices of size $9^m\times 9^m$ we iterate the construction of the preceding proposition:  We form $$Q := [\underbrace{(Q'+I_9) \otimes \ldots \otimes (Q'+I_9)}_{m \textnormal{ times}}] - I_{9^m}.$$  By considering the Kronecker product of matrices, we see that the off-diagonal entries of $Q$ are cube 
 roots of unity, and hence $Q$ is a cube root signature matrix.  
 Furthermore, we see that $Q$ is nontrivial, because the $9\times9$ matrix $Q'$ has a nontrivial entry $\omega$
 in addition to $1$'s, so the Kronecker product with itself produces at least one nontrivial entry.
\end{proof}

\begin{remark}
One can see that Theorem~\ref{bootstrap-9-thm} shows there is an infinite number of nontrivial cube root signature matrices
of arbitrarily large size.
\end{remark}

$ $

\noindent \textsc{Table 1.} Possible $(n,k)$ values for cube root signature matrices with $2\le k<n \leq 100$.

\smallskip

\begin{center}
  \begin{tabular}{|c|c|c|c|c|c|c|}
\hline
$\boldsymbol{n}$ & $\boldsymbol{k}$ & $\boldsymbol{\mu}$ & $\boldsymbol{e}$ & $\boldsymbol{\lambda_1}$ & $\boldsymbol{\lambda_2}$ & \textbf{Do they exist?} \\
\hline
\hline
9 & 6 & -2 & 3 & -7 & 5 & \text{Yes.  (Theorem~\ref{9x9-sig-thm})} \\
\hline
33 & 11 & 4 & 9 & -4 & 8 & \text{Unknown.} \\
\hline
36 & 21 & -2 & 12 & -7 & 5 & \text{Unknown.} \\
\hline
45 & 12 & 7 & 12 & -4 & 11 & \text{Unknown.} \\
\hline
51 & 34 & -5 & 18 & -10 & 5 &\text{Unknown.} \\
\hline
81 & 45 & -2 & 27 & -10 & 8 & \text{Yes.  (Theorem~\ref{bootstrap-9-thm})} \\
\hline
96 & 76 & -14 & 36 & -19 & 5 & \text{Unknown.} \\
\hline
99 & 33 & 7 & 30 & -7 & 14 & \text{Unknown.} \\
\hline
  \end{tabular}
\end{center}

$ $

$ $

\noindent This table contains all possible $(n,k)$ values for nontrivial cube root signature matrices, as determined by the algorithm in Section~\ref{algorithm-sec}.  Thus, if $Q$ is a nontrivial cube root signature matrix of an equiangular $(n,k)$-frame with $n \leq 100$,  then $n$ and $k$ must be the values in one of the rows of this table.  It is unknown whether there exist nontrivial signature matrices for all of the rows in the table.

\providecommand{\bysame}{\leavevmode\hbox to3em{\hrulefill}\thinspace}

\end{document}